\documentclass[11pt,reqno]{amsart}
\usepackage{xifthen,setspace}
\usepackage{pkgfile}
\newtheorem{setup}{Setup}

\newtheorem{definition}{Definition}

\newcommand{\e}{{\epsilon}}

\newcommand{\DD}{{\mathbb{D}}}
\newcommand{\nin}{{\noindent}}
\newcommand{\dd}{{\delta}}
\newcommand\old[1]{}
\title{Convergence of discrete Green functions with Neumann boundary conditions}

\author{Shirshendu Ganguly and Yuval Peres}
\address{Department of Mathematics\\ University of Washington\\ WA, USA.}
\email{sganguly@math.washington.edu}
\address{ Microsoft Research\\ WA, USA.}
\email{peres@microsoft.com}

\def\Z{\mathbb{Z}}

\begin{document}
\maketitle

\begin{abstract}
In this note we prove convergence of Green functions with Neumann boundary conditions for the random walk  to their continuous counterparts. Also a few  Beurling type hitting estimates are obtained for the random walk on discretizations of smooth domains. These have been used recently in the study of a two dimensional competing aggregation system known as \emph{Competitive Erosion}. Some of the statements appearing in this note are classical for $\mathbb{Z}^2$. However additional arguments are needed for the proofs in the bounded geometry setting. 
\end{abstract}
\section{Introduction}\label{intro}
\nin
The article considers simple random walk on discretizations of smooth planar simply connected domains.  It has two parts. In the first part we show convergence of a class of Green functions with Neumann boundary conditions for the random walk to the corresponding one for Reflected Brownian motion on the same domain.  Such  convergence results are typically technically challenging and are often extremely useful while understanding scaling limits of statistical physics models. For e.g. results in \cite{hk} were used in understanding  a modified version of Diffusion Limited Aggregation (DLA). More recently convergence of ``discrete analytic functions" to their continuous counterparts were used to show convergence and conformal invariance of well known critical two dimensional statistical physics models. See \cite{smir} and the references therein for a detailed account of such results.\\
The  results appearing in this article have been used recently in \cite{ero} to establish \emph{conformal invariance} of a competing aggregation system known as {\emph{Competitive Erosion}} on smooth domains.\\
\nin
Most of the proofs  rely on \cite{BC}, \cite{wtf} and \cite{thes} where random walk on  discretizations of regular domains was shown to converge to Reflected Brownian motion under suitable time change and subsequent local CLT estimates were obtained. \\
We remark that similar results appear in \cite{hk}, in the setting of the Dirichlet problem.\\
\nin
In the second part we prove some Beurling type and other hitting estimates for random walk in this setting.
Most of the results in this part are classical when the underlying lattice is the whole of $\mathbb{Z}^2.$ However for  bounded geometry, additional arguments involving heat kernel estimates for the random walk are required. 
\subsection{Informal set up}
Given any  simply connected domain $\U \subset \C$ with certain regularity properties, consider its discretization $\U_n:=\U \cap \frac{1}{n}\mathbb{Z}^2$. Abusing notation a little we denote by $\U_n$ the graph where the edges are induced by the nearest neighbor edges on $\frac{1}{n}\mathbb{Z}^2$ (formal definition appears later).  Fix two points $x_1,x_2$ on $\U_n$ near the boundary of $\U$ and consider the  function on $\U_n$ which is harmonic on all points except $x_1,$ and $x_2$ where the discrete laplacian is $1$ and $-1$ respectively. (For the formal definition of laplacian see \eqref{lapnot1}). The goal of this note is to understand convergence of these functions as the mesh size goes to $0.$ However for technical purposes we study `smoothed' versions of such functions whose laplacian vanishes except on small open sets near the points $x_1$ and $x_2$. More precisely, fix a small $\dd>0$ and choose the open sets to be  balls of radius roughly $\dd$ and also at distance $\dd$ from $x_1$ and $x_2$, (see Fig \ref{blobdef234}). We consider the function on $\U_n$ which is harmonic at every lattice point outside the balls and on the balls the laplacian is roughly the inverse of the the number of lattice points inside the ball. Thus the singularities now are uniformly distributed over open sets in the interior instead of being at points. Formal definition appears in the next section.  
Such functions were useful in the recent study a competing aggregation system known as \emph{Competitive Erosion}. For more details see \cite{ero}.
We remark that the proof techniques in this article are general and should work for a much more general class of laplacian conditions.

\section{Formal definitions and setup}\label{precise}
\nin
$\C$ will denote the complex plane. For any two points $x,y \in \C,$ $d(x,y)$ will denote the euclidean distance between them. Also for any set $A\subset \C$ and any $x \in \C$ denote by $d(x,A)$, the distance between the point and the set. $\DD$ will be used to denote the unit disc centered at the origin in the complex plane. \\
\noindent
\begin{definition}\label{bnddef1}
For any domain $B\subset \C$ denote by $\partial B$ the boundary of $B.$ Also for any graph $G=(V,E)$ with vertices $V$ and edges $E,$ for any $A\subset V$ let 
\begin{equation}\label{bdry1}
\partial_{out}A :=\{y \in A^c:\,\,\exists\,\, x\in A\,\,\text{such that }  x\sim y\}.
\end{equation}
\end{definition}
\nin
Let  a bounded simply connected domain $\U \subset \C$ is ``smooth" mean that  the boundary of $\U$ is an analytic curve (equivalently the conformal map from $\U$ to $\DD$ has a conformal extension across the boundary, see \cite[Prop 3.1]{pom}). From now on all our domains will be bounded, simply connected and smooth. Hence we will drop the adjectives for brevity. 

\noindent
\begin{setup}\label{para} 
Given $\U$ we take $\U_n = \U \cap (\frac1n \Z^2),$ as our vertex set. As the edges of our graph we take the usual nearest-neighbor edges of $\U_n$ though of as a subset of $\frac1n \Z^2$. However we delete every such edge which intersects $\U^c$. By the smoothness assumption on $\U$, $\U_n$ will be connected for large enough $n$. See Remark \ref{graphcon123} below.\\

\nin
Fix $x_1,x_2 \in \partial \U$. For small enough $\dd>0$ let $y_1, y_2 \in \U$ be such that,
 \begin{eqnarray*}
 d(x_{i},y_i)&=& \dd \\
 d(y_{i},\partial \U) &> & \dd/2.
  \end{eqnarray*}   
For $i=1,2,$ let $\U_i=B(y_i,\frac{\dd}{4})$ (we will call them 'blobs').   
As discrete approximations of $\U_i$  we take $$\U_{i,n}=B(z_{i,n}, \frac{\dd}{4})\cap \U_{n},$$ where $z_{i,n}\in \frac{1}{n}\Z^2$ is the closest lattice point to $y_i$.  
\end{setup}
\nin
 Note that in the above, $y_i$'s were just required to satisfy certain  properties and other than that were completely arbitrary. Also we abuse notation a little in the definition of the blobs: $n$ should be thought of as large and hence $\U_n$ (the underlying graph) should not be confused with the blobs $\U_1$ and $\U_2$. 
\begin{remark}\label{graphcon123}
The smoothness assumption on $\U$ allows us to choose the $y_i$'s.  This is formally proved in Corollary \ref{ias1}. See Fig. \ref{blobdef234}.
The connectedness of $\U_n$ for large enough $n$ follows due to the locally half plane like behavior, see \eqref{localhalf}.  
\end{remark}

\begin{figure}
\centering
\includegraphics[scale=.8]{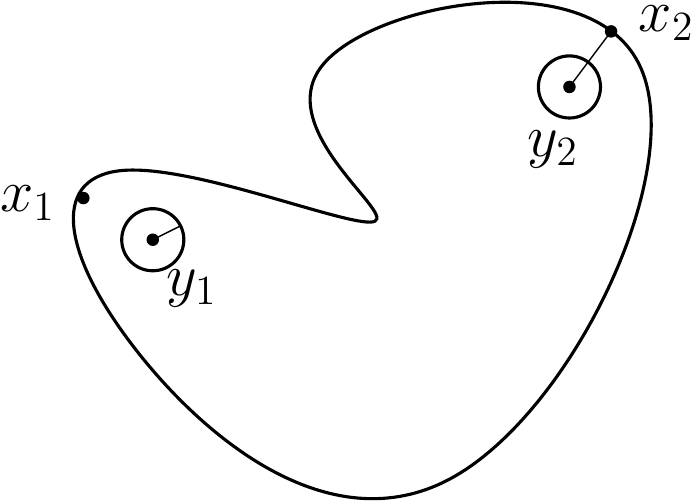}
\caption{$y_i$'s are points at distance $\dd$ from $x_i$'s. They are at distance at least $\frac{\dd}{2}$ from $\partial \U$. The blobs are discs of radius $\frac{\dd}{4}$ centered at the $y_i$'s.}
\label{blobdef234}
\end{figure}

\begin{remark}\label{bili1}
\noindent
For the domain $\U$ and points $x_1,x_2 \in \partial \U$ let  
\begin{eqnarray}
\label{confmap1}
\phi:\DD \rightarrow \U\\
\nonumber
\psi: \U  \rightarrow \DD
\end{eqnarray}
be conformal maps such that $\phi \circ \psi,$ $\psi \circ\phi$ are the identity maps on the respective domains and
$\psi(x_1) = -i,\, \psi(x_2) =  i.$ 
The existence of such maps is guaranteed by the Riemann Mapping Theorem. See for eg: \cite[Chapter 6]{ahl1}. In fact there exists a family of such pairs since a conformal map between domains has three degrees of freedom and here we have fixed the value at only two points.  However we choose a particular pair $(\phi,\psi)$ assumed to be fixed throughout the rest of the article.
Since $\U$ is smooth, using Schwarz reflection  $\phi$ and hence $\psi$ can be extended conformally  across the boundary onto some neighborhoods of $\overline \DD$ and $\overline \U.$ In particular this implies that the derivatives $|\phi'|$ and $|\psi'|$ are bounded away from $0$ and $\infty$ on $\overline \DD$ and $\overline \U$  respectively. See \cite[Prop 3.1]{pom}. This bi-Lipschitz nature of the maps will be used in several distortion estimates throughout the rest of the article.
Also note,  by construction $ |\U_{1,n}|=|\U_{2,n}|.$  This will be technically convenient.
\end{remark}

\subsection{Assumptions, notations and conventions}\label{techass} We summarize some of the notations already used and introduce some new notations and conventions to be used in the sequel.

\noindent
Through out the article, by random walk on $\U_n$, we will mean the 
continuous time random  walk  with  $\exp(2n^2)$  waiting times (mean $\frac{1}{2n^2}$) unless specifically mentioned otherwise.
This is done to ensure that the random walk density converges to that of Reflected Brownian motion. A fact which would be used heavily.\\
\nin
We will denote the complex plane by $\C.$
For any two points $x,y \in \C,$  $d(x,y)$ will be used  to denote the euclidean distance between them. Also for any set $A\subset \C$ and any $x \in \C$ denote by $d(x,A)$, the distance between the point and the set.  
$B(x,\e )$ denotes the open euclidean ball of radius  $\e$ with center $x$.
For any process, and a subset $A$ of the corresponding state space,
$\tau(A)$ will denote the hitting time of that set (we drop the dependence on the process in the notation since it will be clear from context). Also $\mathbf{1}(\cdot)$ will be used to denote the indicator function.\\
\noindent
To avoid cumbersome notation, we will often use the same letter (generally $C$, $D$, $c$ or $d$)  for a constant whose value may change from line to line. $O(\cdot),\Omega(\cdot),\Theta(\cdot)$ are used to denote their usual meaning.

\section{PART I: Green Function: definitions and results}\label{TP}

\subsection{Discrete Green Function}
Recall that we consider continuous time random walk on $\U_n$  with exponential waiting times with mean $\frac{1}{2n^2}$ (Section \ref{techass}). Call it $X(t).$ For $x,y \in \U_n$  let
\begin{equation}\label{rwhk}
\mathbb{P}_x(X(t)=y)
\end{equation}
denote the chance that the random walk on $\U_n$ starting from $x$ is at $y$ at time $t.$ For notational simplicity we suppress the $n$ dependence in $\P$ since the graph will be clear from context.
Similarly for any set $A\subset \U_n$, let $\mathbb{P}_x(X(t)\in A)$ denote the chance that the random walk is in $A$ at time $t$ starting from $x$.
\begin{definition}(Green function) Define the function $G_n$ on $\U_n$ :
 for any $x\in \U_n,$
 \begin{equation}\label{dgf1}
 G_n(x):=\frac{2n^2}{|\U_{1,n}|}\int_{0}^{\infty}[\mathbb{P}_x(X(t)\in \U_{1,n})-\mathbb{P}_x(X(t)\in \U_{2,n})]dt-c.
 \end{equation}
where $c=c(\dd)$ is some constant explicitly mentioned in \eqref{specificconst}. 
\end{definition}
\nin
The dependence of $G_n$ on $\dd$ (through $\U_{1,n},\U_{2,n}$) is suppressed in the notation. 
The centering constant $c$ is not important. The only purpose of the centering is to ensure that for small $\dd$ if $n$ is very large then the Green function upto a universal multiplicative constant (domain independent) approaches the function $\log\left|\frac{\psi(x)-i}{\psi(x)+i}\right|, $ (where $\psi$ was defined in \eqref{confmap1}). The convergence results are stated in Section \ref{mr}.\\

\nin
Observe that the Green function upto translation is the difference in the amount of time random walk spends in $\U_{1,n}$ and in $\U_{2,n}$ respectively. It is shown later that the ``discrete laplacian" (see \eqref{lapnot1}) of $G_n$ is 
$$\frac{1}{| \U_{1,n}|} (\mathbf{1}( \U_{1,n})-\mathbf{1}( \U_{2,n})).$$ Thus as $\dd$ goes to $0$, the functions $G_{n}$ can be thought of as 'smoothly' approximating the function on $\U_n$ which is discrete harmonic on $\U_n$ and has laplacian $1$ and $-1$ at $x_1$ and $x_2$ respectively. As discussed in the beginning, the purpose of this article is to prove convergence of such `smoothed' Green functions.

\nin
The fact that the integral in the expression for $G_n$ is absolutely integrable follows from the following lemma.
\begin{lem}\label{welldefined} Given $\U$ as in Setup \ref{para} there exists a constant $D=D(\U)$ and a time $T=T(\U)$ such that, for all large enough $n$,
$$\sup_{x\in \U_n}|\mathbb{P}_x(X(t)\in \U_{1,n})-\mathbb{P}_x(X(t)\in \U_{2,n})|\le 2e^{-Dt}$$ 
for all $t\ge T$.
\end{lem}

\begin{proof}
Let us consider the random walk on $\U_n$. The following is a standard consequence of the sub-multiplicative nature of the worst case total variation norm $d_{TV}(\cdot)$: 
\begin{equation}\label{subm}
d_{TV}(\ell t_{mix}(1/4))\le 2^{-\ell}
\end{equation} 
where $t_{mix}(1/4)$ is the $\frac{1}{4}$ total variation mixing time, (see \cite[(4.34)]{lpw}). 
Let $\pi_{RW}$ be the stationary measure for the random walk on $\U_n.$ It is well known and easy to verify that for any $z\in \U_n,$ 
\begin{equation}\label{stationmeas}
\pi_{RW}(z)=\frac{d_z}{\sum_{z'\in \U_n}d_{z'}}
\end{equation}
 where $d_z$ is the degree of the vertex $z$ in $\U_n.$ Recall from Setup \ref{para} that,
  $$|\U_{1,n}|= |\U_{2,n}|$$ and also that all points in $\U_{1,n}$ and $\U_{2,n}$ are in the interior of $\U$ and hence have four neighbors for large $n$. Thus 
$$\pi_{RW}(\U_{1,n})=\pi_{RW}(\U_{2,n}).$$  
By definition for any $t>0,$  $$\sup_{{x\in \U_n}\atop {i=1,2}}|\mathbb{P}_x(X(t)\in \U_{i,n})-\pi_{RW}(\U_{i,n})|\le d_{TV}(t).$$
Thus uniformly over $x\in \U_n,$
\begin{align*}
|\mathbb{P}_x(X(t)\in \U_{1,n})-\mathbb{P}_x(X(t)\in \U_{2,n})|&\le  |\mathbb{P}_x(X(t)\in \U_{1,n})-\pi_{RW}(\U_{1,n})|+|\mathbb{P}_x(X(t)\in \U_{2,n})-\pi_{RW}(\U_{2,n})|\\
&\le  2d_{TV}(t).
\end{align*}
The result now follows from \eqref{subm} and the standard fact that 
$t_{mix}(1/4)=O(1)$ on $\U_n$, (see Lemma \ref{lmt}). 
\hfill \end{proof}

\noindent
For any function $f:\U_n \rightarrow \mathbb{R}$ define the laplacian $\Delta f: \U_n \rightarrow \mathbb{R}$, where for any $x \in \U_n,$ 
 \begin{equation}\label{lapnot1}
 \Delta f(x):=f(x)-\frac{1}{d_x}\sum_{y\sim x}f(y),
\end{equation}
($d_x$ is the degree of the vertex $x$ and $y\sim x$ denotes that $y$ is a neighbor of $x$).
\begin{lem}Consider the function $G_n(\cdot)$ on $\U_n.$ Then,
\begin{equation}\label{laplacian}
\Delta(G_n)=\frac{1}{| \U_{1,n}|} (\mathbf{1}( \U_{1,n})-\mathbf{1}( \U_{2,n})),
\end{equation}
where for any subset $A\subset \U_n$, $\mathbf{1}(A)$ denotes the indicator of the set $A$.
\end{lem}
\begin{proof}Proof follows from definition of $G_n$ and looking at the first step of random walk  which by definition is of expected duration $\frac{1}{2n^2}.$  
Thus we have  $$G_n(x)=\frac{1}{|\U_{1,n}|} (\mathbf{1}( \U_{1,n})-\mathbf{1}( \U_{2,n}))(x)+\frac{1}{d_x}\sum_{y\sim x}G_n(y)$$
and hence the lemma.
\end{proof}
\subsection{Main Result}\label{mr}
In this section we state  the main convergence result of this paper.
We start by defining the following function on the domain $\U,$ (recall the sets $\U_1,\U_2$ from Setup \ref{para}):
\begin{equation}\label{laplace}
\tilde f:=\frac{16}{\rm{ area}(\U_{1})}\bigl(\mathbf{1}(\U_{2})-\mathbf{1}(\U_{1})\bigr).
\end{equation}
$16$ is a constant that falls out of some natural integrals involving the heat kernel of Reflected Brownian motion and is not important. One could normalize things to make the constant $1$, however we choose not to do that.\\

\noindent
Recall the functions $\phi$ and $\psi$ from \eqref{confmap1}.
Let $\psi(\U_{i})=:A_i$. 
Then, 
\begin{equation}\label{mappedsource}
\tilde f\circ\phi=\frac{16}{\rm{ area}(\U_{1})}\bigl(\mathbf{1}(A_2)-\mathbf{1}(A_1)\bigr),
\end{equation}
is a function on $\DD.$
For any $\U$ as in Setup \ref{para}, define the function $G_{*}:\overline \U \rightarrow \mathbb{R}$ such that for all $z \in \overline \U$ if $y\in \overline \DD$ is such that $\phi(y)=z$,  then 
\begin{equation}\label{transformed}
G_{*}(z)=\frac{1}{\pi}\int_{|\zeta|<1}\tilde f\circ \phi(\zeta)|\phi'(\zeta)|^2\log(|\zeta-y)(1-\bar{\zeta}y)|^2)d\xi d\eta,
\end{equation}
where $\zeta=\xi+i\eta.$  Recall that by the smoothness assumption on $\U$ the maps $\phi$ and  $\psi$ have extensions across the boundaries of $\DD$ and $\U$ respectively and hence $G_*$ can be defined on $\overline \U.$
 Notice the dependence of $G_*$ on $\dd$ through $\tilde f$. However for brevity we choose to suppress the dependence on $\dd$ (see Setup \ref{para}) in the notation. \\

\nin
Even though the function $G_n$ in \eqref{dgf1} is defined on the graph $\U_n,$ to state the next result, we use linear interpolation  to think of it as a function on the closure of the whole domain, $\overline \U$. 
We are now ready to state one of the main results of this paper.
The interpolation scheme is defined precisely in Subsection \ref{conv1} where the proof of the result appears. Informally it is done in the following way:
\begin{itemize}
\item [(i).] Extend the function $G_{n}$ from $\U_n$ to $\frac{1}{n}\mathbb{Z}^2$ by fixing it to be $0$ outside $\U_n.$
\item [(ii).] Extend the function $G_{n}$ to all the edges of $\frac{1}{n}\mathbb{Z}^2$ by linearly interpolating the values on the vertices.
\item [(iii).] Extend it to each face of $\frac{1}{n}\mathbb{Z}^2$ so that it is a harmonic function on each face given the value on the edges. Thus the function is extended to the entire complex plane $\mathbb{C}.$  By abusing notation a little we still denote the extended function by $G_n$ as well.
\end{itemize}
\begin{thm}\label{convergence} For all small enough $\dd,$
$$\lim_{{m\rightarrow \infty}\atop{n=2^m}}\sup_{z\in \overline{\U}}|G_n(z)-G_{*}(z)|=0.$$
\end{thm}
\noindent
Before proving the above, we remark (see \cite[Lemma 5.3]{ero}) that as $\dd$ goes to $0$ the function $G_*(\cdot)$ approaches (up to an explicit multiplicative constant) the function  \begin{equation}\label{recall1}
 \log\left|\frac{\psi(\cdot)-i}{\psi(\cdot)+i}\right|.
\end{equation}

\section{Proof of Theorem \ref{convergence}}\label{conformcon}
\nin The proof of Theorem \ref{convergence} involves developing some tools using convergence of random walk on $\U_n$ to Reflected Brownian motion on $\overline\U.$ For a formal definition of Reflected Brownian motion on $\overline\U$ see \cite[Definition 2.7]{wtf}. Also see \cite{bass,chendef,BC}. Throughout the rest of the article we will denote it by $B_t.$
Consider the function, for $z \in \overline \U,$ 
\begin{equation}\label{discdef}
G(z):=\frac{2}{\rm{ area}(\U_{1})}\int_{0}^{\infty}[P_{z}(B_{t} \in \U_1)-P_{z}(B_{t} \in \U_2)]\,dt-c.
\end{equation}  where for $i=1,2,$ $P_{z}(B_{t} \in \U_i)$ denotes the probability that started from $z,$ $B_t$  is in $\U_i.$ 
\nin
The constant $c$ is chosen such that the integral of $G$ along $\partial \U$ is $0$ where we  parametrize the boundary $\partial U$ by $\theta\in [0,2\pi)$ via the conformal map $\phi$ \eqref{confmap1}. Formally we fix $c$ such that 
\begin{equation}\label{specificconst} \int_{|\zeta|=1} G\circ \phi (\zeta)\frac{d\zeta}{\zeta}=0.
\end{equation}

\noindent
Compare the expression of $G(\cdot)$ with $G_n(\cdot)$ defined in \eqref{dgf1}  (note that the constant $c$ is the same in both expressions). \\
\nin
Before providing formal arguments we sketch the general outline of the proof first. The proof of  Theorem \ref{convergence} has two parts: in the first part we show that $G_{n}$ converges to $G.$ This will follow by convergence of the random walk measure on $\U_n$ to Reflected Brownian motion $B_t.$ For more on this see \cite{BC}, \cite{wtf} and the references therein. Thus the only remaining step then is to show that indeed  $$G_*=G.$$ This will be proved using the fact that the density for Reflected Brownian motion is a fundamental solution to the Neumann problem and hence the function $G$ roughly satisfies,
\begin{equation}\label{roughpde}
\left(\frac{\partial^2}{\partial x^2}+\frac{\partial^2}{\partial y^2}\right)  G = \frac{4}{\rm{area}(\U_1)}(\mathbf{1}(\U_2)-\mathbf{1}(\U_1))
\end{equation}
 with Neumann boundary conditions. 
One then checks that $G_*$ is a solution to the above Neumann problem as well. The proof is then complete by uniqueness of the solution of such a problem which allows us to conclude that $G_*=G$.
We adopt the following standard notation: 
\begin{equation}\label{contlap1}
\Delta :\equiv\frac{\partial^2}{\partial x^2}+\frac{\partial^2}{\partial y^2}.
\end{equation}
Recall \eqref{lapnot1}.
Thus we use $\Delta$ to denote the laplacian in both the continuous and discrete setting since there will be no scope of confusion. 

 \subsection{Continuum version of $G_n$}\label{sub:contver}
We begin by studying the function $G(z)$ defined in \eqref{discdef}.
Let 
$B_{t}$ be Reflected Brownian motion (RBM) on $\overline{\U}$ and  $p(t,x,y)$ be the heat kernel of $B_t$ defined on $$\mathbb{R_{+}\times \overline{\U}\times \overline{\U}},$$ i.e.\, $p(t,x,y)$ is the density of RBM started from $x$ at time $t$ at point $y$.
Before proceeding we state some classical results about regularity properties of $p(t,x,y).$ 
\begin{thm}\label{properties} \noindent
\begin{itemize}
\item [a.]\label{continuity}\cite[Lemma 2.1]{sato} $p(t,x,y)$ is continuous on $(0,\infty)\times \overline\U \times \overline\U.$ 
\item [b.]\label{continuity1}\cite[Theorem 2.1]{sato} Let $f$ be a compactly supported $C^{\infty}$ function defined on $\U$. Then $$\int_{\overline\U}p(t,x,y)f(y)dy$$ is continuous on $(0,\infty)\times \overline \U$.
\item [c.]
\label{heatequation}\cite[Theorem 2.2]{sato} Let $f$ be a compactly supported $C^{\infty}$ function defined on $\U$. Then $$u(t,x)=\int_{0}^{t}ds\int_{\overline\U}p(s,x,y)f(y)dy$$ has the following properties :
\begin{itemize}
\item [i.]is continuous on $(0,\infty)\times \overline\U$, continuously differentiable in $t$ in $(0,\infty)$ and of class $C^{2}(\U)$ and $C^{1}(\overline \U)$ as a function of $x$,\\
\item [ii.]$\left(\frac{\partial}{\partial t}-\frac{1}{2}\Delta\right)u(t,x)=f(x),$\\
\item [iii.]$\frac{\partial}{\partial \nu}u(t,x)=0,$ where $\frac{\partial}{\partial \nu}$ denotes the normal derivative and $x \in \partial \U,$\\
\item [iv.]$\displaystyle{\lim_{t\rightarrow0}u(t,x)=0}$ uniformly on $\overline\U.$
\end{itemize}
\end{itemize}
\end{thm}
\nin 
The results quoted from \cite{sato} are actually proved in much more generality. However for our purposes the above versions would suffice.
\nin
For $i=1,2,$ recall $\U_i$ from Setup \ref{para} and let, 
\begin{equation}\label{abbre12}
p(t,x,\U_i):=\int_{\U_i}p(t,x,y)dy.
\end{equation}
Thus the expression in  \eqref{discdef} is the same as, 
\begin{equation}\label{rewrite1}
G(z)=\frac{2}{\rm{ area}(\U_{1})}\int_{0}^{\infty}[p(t,z,\U_{1})-p(t,z,\U_{2})]\,dt-c.
\end{equation}
Recall that the above expression depends on $\dd$ (Setup \ref{para}) which determines the sets $\U_1$ and $\U_2$.
However we will suppress the dependence on $\dd$ for notational brevity since there is little chance for confusion. 

\begin{remark}\label{contint}
The fact that the above integral is absolutely convergent directly follows from the following mixing lemma and the fact that by choice $\rm{area}(\U_1)=\rm{area}(\U_2)$. 
\end{remark}

\begin{lem}\label{bh}\cite[Theorem 2.4]{BassHsu}  There exists constants $c=c(\U),T=T(\U)>0$ such that for all $x,y\in \overline{\U}$ and $t>T,$
$$|p(t,x,y)-\frac{1}{\rm{ area }(\U)}|\le e^{-ct}.$$
\end{lem}
\nin
Note that to be able to use Theorem \ref{properties} to prove \eqref{roughpde}  one has to approximate the indicator functions on $\U_1,\U_2$ by $C^{\infty}$ functions.
Let $g_{1},g_{2}\ldots$ be $C^\infty$ functions taking values in the interval $[0,1]$ such that for any $j \in \mathbb{N}$ and $z \in \mathbb{C}$
\begin{eqnarray*}
g_{j}(z)=\left \{\begin{array}{cc}
1 & |z|\le 1-\frac{1}{j}\\\\
0 & |z|\ge 1-\frac{1}{2j}.
\end{array}
\right.
\end{eqnarray*}
That is $g_j$'s form a sequence of smooth functions approximating from below the indicator function on the unit ball. 
For any $a\in \mathbb{C},$ $b>0$ and $j\in \mathbb{N}$ we denote by $g_{j}(a,b,\cdot),$ the function such that for any $z \in \mathbb{C}$ $$g_{j}(a,b,z)=g_{j}(\frac{z-a}{b}),$$
i.e. $g_{j}(a,b,\cdot)$ approximates the indicator function on the ball $B(a,b).$
\noindent
Recall from Setup \ref{para} that $\U_1,\U_2$ have centers $y_1,y_2$ and radius $\frac{\dd}{4}.$
For brevity let $\tilde{\dd}:=\frac{\dd}{4}$.
The next easy lemma uses the $g_j$'s to approximate the integrals appearing in \eqref{rewrite1}.
\begin{lem}\label{smoothapprox}Given any $T$ for every $\e>0$ there exists $J$ such that for all $j>J$ and $i=1,2$
$$\sup_{z\in \overline \U}\left|\int_{0}^{T}\bigl[\int_{\overline \U}p(t,z,\zeta)g_j(y_i,\tilde \dd,\zeta)d\xi d\eta-p(t,z,\U_i)\bigr]dt\right|\le \e$$ where $\zeta=\xi+i\eta.$ 
\end{lem}
\begin{proof}  
We prove it only for the case $i=1.$
Using the trivial observation that for all $t>0$, both $$p(t,z,\U_1) \mbox{ and } \int_{\overline \U}p(t,z,\zeta)g_j(y_1,\tilde \dd,\zeta)d\xi d\eta  \le 1,$$  and $p(t,z,\U_1)=\int_{\U}p(t,z,\zeta)\mathbf{1}(\zeta\in \U_1)d\xi d\eta$ it suffices to show ,
$$\sup_{z\in \overline \U}\left|\int_{\e/{2}}^{T}\bigl[\int_{\overline \U}p(t,z,\zeta)\left(g_j(y_1,\tilde \dd,\zeta)-\mathbf{1}(\zeta\in \U_i)\right )d\xi d\eta\bigr]dt\right|\le \frac{\e}{2}.$$
Now by Lemma \ref{continuity}, $p(t,z,\zeta)$ is bounded on the cylinder $[\e/2,T]\times \overline {\U}\times \overline {\U}$.  Also clearly,
\begin{equation}\label{l1conv}
\lim_{j\rightarrow \infty}\int_{\overline {\U}}|g_j(y_1,\tilde \dd,\zeta)-\mathbf{1}(\zeta\in \U_1)|d\xi d\eta=0.
\end{equation}
Thus we are done.
\end{proof}
\nin
In the next couple of lemmas we approximate the Green function \eqref{rewrite1} in terms of the functions $g_j(\cdot,\cdot,\cdot).$ 

\begin{lem}\label{approx1}
For $i=1,2,$ given $\e>0,$ there exists  $T=T(\e)$ such that for all $t\ge T$ and all large enough $j$,
$$
\sup_{z\in \overline \U}\left|\int_{\overline \U}p(t,z,\zeta)g_j(y_i,\tilde \dd,\zeta)d \xi d\eta- \frac{\rm{area}(\U_{1})}{\rm{area}(\U)}\right|\le \e, 
$$
where $\zeta=\xi+i\eta.$
\end{lem}

\begin{proof}
By Lemma \ref{bh} for large $t$ and all $j$,
\begin{equation}\label{exptail}
\sup_{z\in \overline \U}\int_{\overline \U}\left|p(t,z,\zeta)-\frac{1}{\rm{ area}(\U)}\right|g_j(y_i,\tilde \dd,\zeta)d \xi d \eta=O(e^{-ct}). 
\end{equation}
Now by \eqref{l1conv} for large $j,$
$$\frac{1}{\rm{area}(\U)}\left|\int_{\overline \U} g_j(y_i,\tilde \dd,\zeta)d \xi d \eta-\rm{area}(\U_{1})\right|\le \e.$$
Thus the above two statements along with triangle inequality complete the proof.
\end{proof}
\nin
Given $j\in \mathbb{N}$ and $T>0$ define, 
\begin{equation}\label{smoothgreen}
G_{j,T}(z)=\frac{2}{\rm{area} (\U_1)}\int_{0}^{T}dt \left[\int_{\overline \U}p(t,z,\zeta)[g_j(y_1,\tilde \dd,\zeta)-g_j(y_2,\tilde \dd,\zeta)]d \xi d \eta\right]-c,
\end{equation}
where $c$ is the same as in \eqref{discdef} and $\zeta=\xi+i\eta.$. We then have the following lemma showing that $G_{j,T}$ approximate $G$ as $j,T$ go to infinity.
\begin{lem}\label{approxfinite}$$\lim_{j,T\rightarrow \infty } \sup_{z \in \overline \U}|G_{j,T}(z)-G(z)|=0.$$
\end{lem}
\begin{proof} 
Clearly by definition $$\int_{\U}[g_j(y_1,\tilde \dd,\zeta)-g_j(y_2,\tilde \dd,\zeta)]d\xi d\eta=0.$$
Hence for every $j$ and all intervals $[T,T']$ where $T$ is large enough and $T<T'$ by \eqref{exptail}
\begin{equation}\label{tailint2}
\int_{T}^{T'}dt\left|\int_{\overline \U}p(t,z,x)[g_j(y_1,\tilde \dd,x)-g_j(y_2,\tilde \dd,x)]dx\right|=O( e^{-cT}).
\end{equation}
\noindent
Also note that by Lemma \ref{bh} for large enough $T<T',$
\begin{equation}\label{tailint1}
\int_{T}^{T'}|p(t,z,\U_1)-p(t,z,\U_2)|dt=O(e^{-cT}).
\end{equation}
\noindent
Hence given $\e>0,$ choose $T$ such that both the above integrals in  \eqref{tailint2} and 
\eqref{tailint1} are less than $\e$ for all $T'>T$ and $z\in \overline \U$.
 The lemma now follows by Lemma \ref{smoothapprox}.
\end{proof}

\noindent
Since $g_j$'s are compactly supported $C^{\infty}$ functions, for any $j$  and $z$ using Theorem \ref{continuity1} $b.,c.,$ and fundamental theorem of calculus we have,
\begin{equation}\label{timederiv}
\frac{\partial}{\partial t} G_{j,t}(z)=\frac{2}{\rm{area} (\U_1)}\int_{\overline \U}p(t,z,x)[g_j(y_1,\tilde \dd,x)-g_j(y_2,\tilde \dd,x)]dx.
\end{equation}
Also by Theorem \ref{heatequation} $c. ii, iii.$ for any $t>0$ 
\begin{eqnarray}\label{derivelapl}
\nonumber
\Delta G_{j,t}(z)&=&\frac{4}{\rm{area}(\U_1)}[g_j(y_2,\tilde \dd,z)-g_j(y_1,\tilde \dd,z)]+\frac{4}{\rm{area}(\U_1)}\int_{\overline \U}p(t,z,x)[g_j(y_1,\tilde \dd,x)-g_j(y_2,\tilde \dd,x)]dx\\
\frac{\partial}{\partial \nu}G_{j,t}(z) &=& 0 \text{ for all } z\in \partial \U.
\end{eqnarray}

\subsection{Closed form of the limit.}\label{sub:cfl}
In this subsection we show that the function $G$ is same as the function $G_*$ \eqref{transformed}.
\begin{thm}\label{closedform1} Let $\U$ be as in Setup \ref{para}.  For all $z \in \overline \U,$
\begin{equation}\label{cf2}
G(z)=G_*(z).
\end{equation}
\end{thm}
\noindent
To show this,  we identify $G$ as a  solution to a second order differential equation also satisfied by $G_*$. The result then follows by uniqueness of such a solution. We first quote a result in the theory of boundary value problems with Neumann boundary condition.
\begin{lem}\cite[Theorem 8]{bca}\label{pde1} A function $w\in C^{2}(\DD)\cap C^{1}(\overline{\DD})$ on the disc satisfying the following properties
\begin{eqnarray*}
\Delta(w) &=& \frac{f}{4}\\
\frac{\partial w}{\partial \nu}&=& \gamma \text{ on } \partial \DD\\
\frac{1}{2\pi i}\int_{\partial \DD}\frac{w(\zeta)}{\zeta}d\zeta &=& d
\end{eqnarray*} where $f \in L^{1}(\mathbb{R},\DD),$ $\gamma\in C(\mathbb{R},\partial \DD) $ and $d \in \mathbb{R}$ exists iff 
$$
\frac{1}{2\pi i}\int_{|\zeta|=1}\gamma(\zeta)\frac{d\zeta}{\zeta}=\frac{2}{\pi}\int_{|\zeta|<1}f(\zeta)d\xi d\eta.
$$
The unique solution in that case is given by the following 
$$w(z)= d-\frac{1}{2\pi i}\int_{|\zeta|=1}\gamma(\zeta)\log(|\zeta-z|^2)\frac{d\zeta}{\zeta}+\frac{1}{\pi}\int_{|\zeta|<1}f(\zeta)\log(|(\zeta-z)(1-\bar{\zeta}z)|^2)d\xi d\eta ,$$
where $\zeta=\xi+i\eta.$
\end{lem}
\nin
Recall the maps $\phi$ and $\psi$ from \eqref{confmap1}. They will be used throughout the rest of the subsection. The next simple lemma shows how the laplacian changes under a change of variable.

\begin{lem}\label{pushforward} For a function $u\in C^{2}(\U)\cap C^{1}(\overline{\U})$ with the following properties
\begin{align*}
\Delta(u) = f \text{ on } \U\\
\frac{\partial u}{\partial \nu} =  0\, \text{ on } \partial \U
\end{align*} 
the function $v=u\circ \phi$ on $\overline \DD$ satisfies the following properties,\\\\
$i.$ $v\in C^{2}(\DD)\cap C^{1}(\overline{\DD})$.\\\\
$ii.$ $\Delta(v) = (f\circ \phi )|\phi'|^2$.\\\\
$iii.$ $\frac{\partial v}{\partial \nu} =  0 \text{ on } \partial \DD$.
\end{lem}
\begin{proof}
$i.$ is obvious since it is the composition with a conformal map which is analytic across the boundary by hypothesis.\\
$ii.$ follows from chain rule for differentiation and using the Cauchy-Riemann equations.  \\
$iii.$ Since conformal maps preserve angles it follows that
$$\frac{\partial v}{\partial \nu}=\frac{\partial u}{\partial \nu}|\phi'(z)|.$$ 
Now by hypothesis
$$\frac{\partial u}{\partial \nu}\mid_{\partial \U}=0$$ and hence we are done. 
\end{proof}

\nin
Using the above observation we state a lemma analogous to Lemma \ref{pde1} for $\U$.
\begin{lem}\label{pde2} Let function $w\in C^{2}(\U)\cap C^{1}(\overline{\U})$ be a function on $\U$ satisfying the following properties
\begin{eqnarray*}
\Delta(w)&=&\frac{f}{4}\\
\frac{\partial w}{\partial \nu}& = & 0 \text{ on } \partial \U\\
\end{eqnarray*} where $f \in L^{1}(\mathbb{R},\U)$ then  
\begin{equation}\label{rep1}
w\circ\phi(z)= d+\frac{1}{\pi}\int_{|\zeta|<1}f\circ\phi(\zeta)|\phi'|^2(\zeta)\log(|(\zeta-z)(1-\bar{\zeta}z)|^2)d\xi d\eta,
\end{equation}
where $$\frac{1}{2\pi i}\int_{\partial \DD}\frac{w \circ \phi(\zeta)}{\zeta}d\zeta = d.
$$
\end{lem}
\begin{proof} The proof  follows from Lemmas \ref{pde1} and \ref{pushforward}.  
\end{proof}

\nin
We are now ready to prove Theorem \ref{closedform1}.\\

\noindent
\textbf{Proof of Theorem \ref{closedform1}.} For any $j$ and $T$ recall $G_{j,T}$ from \eqref{smoothgreen}.
By \eqref{derivelapl},  $G_{j,T}$ satisfies the hypotheses of Lemma \ref{pde2} with $\Delta G_{j,T}=\frac{f_{j,T}}{4},$ where 
\begin{equation}\label{approxlap}
f_{j,T}(\cdot)  = \frac{16}{\rm{area}(\U_1)}\left[g_j(y_2,\tilde \dd,\cdot)-g_j(y_1,\tilde \dd,\cdot)+\int_{\U}p(T,\cdot,\zeta)[g_j(y_1,\tilde \dd,\zeta)-g_j(y_2,\tilde \dd,\zeta)]d\xi d\eta \right],
\end{equation}
where $\zeta=\xi+ i \eta.$
Thus by Lemma \ref{pde2} for all $z\in \overline \DD,$
\begin{equation}\label{candidate1}
G_{j,T}\circ \phi(z)=c_{j,T}-c+\frac{1}{\pi }\int_{|\zeta|<1}|\phi'|^2\left(f_{j,T}\circ \phi\right) (\zeta)\log(|(\zeta-z)(1-\bar{\zeta}z)|^2)\,d\xi d\eta,
\end{equation}
where 
\begin{equation}\label{express1}
c_{j,T}=\frac{1}{2\pi i}\int_{|\zeta|=1}(G_{j,T}\circ \phi+c)\frac{d \zeta}{\zeta}.
\end{equation}
Recall from \eqref{laplace},
\begin{equation*}
\tilde f=\frac{16}{\rm{ area}(\U_{1})}\bigl(\mathbf{1}(\U_2)-\mathbf{1}(\U_1)\bigr).
\end{equation*}
and that $\psi(\U_{i})=A_i$. 
Then 
\begin{equation*}
\tilde f\circ\phi=\frac{16}{\rm{ area}(\U_{1})}\bigl(\mathbf{1}(A_2)-\mathbf{1}(A_1)\bigr).
\end{equation*}

\noindent
We first prove that 
\begin{equation}\label{suffice1}
\lim_{j,T \rightarrow \infty}\sup_{z\in \overline \DD}|(G_{j,T}\circ \phi(z)-c_{j,T}+c)-G_*\circ \phi(z)|=0.
\end{equation}
Using the expression in \eqref{candidate1} and \eqref{transformed} we get that given $\e>0$, for large enough $j,T$, for all $z\in \overline{\DD}$ 
\begin{align*}
|G_{j,T}\circ \phi(z)-c_{j,T}+c-G_{*}\circ \phi(z)| \le &\frac{32}{\rm{area}(\U_1)}\left[ \int_{\zeta\in A_{1}\setminus  \psi(B(y_1,\tilde \dd (1-\frac{1}{j})))}|\phi'|^2\bigl|\log(|\zeta-z||1-\bar{\zeta}z|)\bigr|d\xi d\eta \right. \\ & + \int_{\zeta\in A_{2}\setminus \psi( B(y_2,\tilde \dd (1-\frac{1}{j}))) }|\phi'|^2\bigl|\log(|\zeta-z||1-\bar{\zeta}z|)\bigr|\,d\xi d\eta  \\ & + \left. \int_{\zeta\in \DD }\e |\phi'|^2\bigl|\log(|\zeta-z||1-\bar{\zeta}z|)\bigr|\,d\xi d\eta \right]. 
\end{align*}
\nin
The terms on the RHS corresponds to difference between the functions $f_{j,T}\circ \phi$ and $\tilde f\circ \phi $. The first term corresponds to difference between $g_j (y_1,\tilde \dd,\phi(\cdot))$ and ${1}(A_1).$ Similarly the second term corresponds to $g_j(y_2,\tilde \dd,\phi(\cdot))$ and ${1}(A_2).$
\nin
The last term corresponds to $\int_{\U}p(T,\cdot,\zeta)[g_j(y_1,\tilde \dd,\zeta)-g_j(y_2,\tilde \dd,\zeta)]d\xi d\eta$. By Lemma \ref{approx1} for any $\e$, for large enough $T$ and $j$,$$\sup_{z\in \overline \U}\left|\int_{\U}p(T,z,\zeta)[g_j(y_1,\tilde \dd,x)-g_j(y_2,\tilde \dd,\zeta )]d\xi d \eta\right|\le \e.$$
Now in all the above terms we can ignore $|\phi'|^2$ since it is bounded by Remark \ref{bili1}. As $\log(|\zeta-z||1-\bar{\zeta}z|)$ is a uniformly locally integrable function in $z$ we see that by making $\e$ small and $j$ large  we can make the above quantity arbitrarily small uniformly over $z\in \overline \DD$.
Thus $$\lim_{j, T\rightarrow \infty}\bigl[\sup_{z\in \overline \DD}|G_{j,T}(z)-c_{j,T}+c-G_{*}(z)|\bigr]=0.$$
Now notice by Lemma \ref{approxfinite} 
$$\lim_{j,T \rightarrow \infty} \frac{1}{2\pi i}\int_{|\zeta=1|}(G_{j,T}\circ \phi)\frac{d \zeta}{\zeta}= \frac{1}{2\pi i}\int_{|\zeta|=1}(G\circ \phi)\frac{d \zeta}{\zeta}.$$
By the choice of $c$ in \eqref{specificconst} 
$$\frac{1}{2\pi i}\int_{|\zeta|=1}(G\circ \phi)\frac{d \zeta}{\zeta}=0.$$
Hence by \eqref{express1}
$$\lim_{j, T\rightarrow \infty}c_{j,T}=c.$$
This along with \eqref{suffice1} completes the proof.

\qed

\begin{remark}\nonumber Recall the constant $c$ in \eqref{discdef} such that.
\begin{equation*} \int_{|\zeta|=1} G\circ \phi (\zeta)\frac{d\zeta}{\zeta}=0.
\end{equation*}
Consider the the special case of the unit disc with point $x_1=-i,x_2=i$ and $y_1=-(1-\dd)i,\,y_2=(1-\dd)i.$ Now owing to the invariance of Reflected Brownian motion on $\DD$ under the transformation
$\zeta \rightarrow -\zeta,$ one sees that the integral in \eqref{discdef}
$$\frac{2}{\rm{ area}(\DD_{1})}\int_{0}^{\infty}[P_{z}(B_{t} \in \DD_1)-P_{z}(B_{t} \in \DD_2)]\,dt$$  is an odd function, (since by the symmetric choice of $y_1,y_2,$ $\DD_1=-\DD_2$).
Thus $c$ is $0$ and 
 \begin{equation*} 
G(z)=\frac{2}{\rm{ area}(\DD_{1})}\int_{0}^{\infty}[P_{z}(B_{t} \in \DD_1)-P_{z}(B_{t} \in \DD_2)]\,dt.
 \end{equation*}
\end{remark}

\subsection{Convergence}\label{conv1}
In this section we prove the last technical piece needed to complete the proof of Theorem \ref{convergence}. Namely we show that the function $G_n$ converges to $G.$ This along with Theorem \ref{closedform1} then would clearly complete the proof of Theorem \ref{convergence}.\\
\noindent
As mentioned in the statement of Theorem \ref{convergence} the function $G_n$ is interpolated from the graph $\U_n$ to  $\overline \U.$ 
We define precisely the interpolation method which is sketched right before the statement of the theorem. We follow the scheme mentioned in \cite[Pf of Theorem 2.12]{wtf} (also appears in \cite[Pf of Theorem 2.2.8]{thes}).\\
\noindent
For $x,y \in \U_n$, define,
\begin{equation}\label{normkernel}
p_n(t,x,y)=\frac{\mathbb{P}_{x}(X(t)=y)}{m_n(y)}
\end{equation}
where $\mathbb{P}_{x}(X(t)=y)$ is the probability that the continuous time random walk $X(t)$  started from $x$ is at $y$ at time $t$ (see \eqref{rwhk}). $m_n(y)=\frac{d_y}{4n^2}$ is the symmetrizing measure of the random walk ($d_y$ is the degree of the vertex $y$).\\
\noindent
For all pairs $x,y\in \frac{1}{n}\mathbb{Z}^2,$ such that at least one of them is not in $\U_n$ define for all $t,$ $p_n(t,x,y)=0.$
Now having defined $p_n(t,x,y)$ for all $t\in \mathbb{R}_+, x,y\in \frac{1}{n}\mathbb{Z}^2$ we extend it to $\mathbb{R}_+\times \overline \U \times \overline \U.$ 
To this end we interpolate $p_n(t,x,y)$ by a sequence of harmonic extensions along simplices, \\

\noindent
$i.$ First harmonically extend it along the edges of $\frac{1}{n}\Z^2$ using the value on the vertices.\\
$ii.$ Then harmonically extend it to the squares (faces of $\frac{1}{n}\Z^2$) using the value on the edges.\\
 Thus we  have extended $p_n(t,x,y)$ to $\mathbb{R}_+\times \C \times \C,$ and hence in particular to $\mathbb{R}_+\times \overline \U \times \overline \U.$\\

\noindent
Note that by \eqref{normkernel} for $x\in \U_n,$ the expression of $G_{n}(x)$ in \eqref{dgf1} is the same as 
 \begin{equation}\label{abbre1234}
\frac{2n^2}{|\U_{1,n}|}\int_{0}^{\infty}\left[\sum_{y\in \U_{1,n}}\frac{p_{n}(t,x,y)}{n^2}-\sum_{y\in \U_{2,n}}\frac{p_{n}(t,x,y)}{n^2}\right]dt -c.
\end{equation}
  $p_{n}(t,x,y)$ is divided by $n^2$ in the above sums since all $y\in \U_{1,n}\bigcup \U_{2,n}$ have degree $4$.\\
\nin
Now by the above extension $p_{n}(t,x,y)$ is extended to all $x\in \overline \U,y \in \U_{1,n}\bigcup \U_{2,n}.$ The above expression thus allows us to define the extended function $G_{n}(x)$ on $\overline\U$. It is easy to see that since the integral in \eqref{abbre1234} is finite for all $x\in \U_n$, the interpolation does not cause additional issues.\\

\noindent 
For notational unification we formally denote the two sums inside the integral in \eqref{abbre1234} as $\mathbb{P}_x(X(t)\in \U_{i,n})$ for $i=1,2$ respectively as in \eqref{dgf1} even though now the function lives on $\overline \U$ and hence $x$ might not be in $\U_n$.\\
 
\noindent 
\nin 
We now state the following lemma which along with Theorem \ref{closedform1} proves Theorem \ref{convergence}.
\begin{lem} \label{convergence1}
$$\lim_{{m\rightarrow \infty} \atop {n=2^m}}\sup_{x\in \overline{\U}}|G_n(x)-G(x)|=0$$
where $G$ is defined in \eqref{discdef}.
\end{lem}
\nin
Before proving the lemma we state a local CLT result for random walk approximation of Reflected Brownian motion. 
\begin{thm}\label{clt}\cite[Theorem 2.12]{wtf} For all positive numbers $a<b$, $$\lim_{{m\rightarrow \infty} \atop {n=2^m}}\sup_{a\le t\le b}\sup_{x,y\in \overline{\U}}|p_n(t,x,y)-p(t,x,y)|=0$$
where $p(t,x,y)$ is the heat kernel of Reflected Brownian motion as defined in Subsection \ref{sub:contver}.
\end{thm} 

\nin
Also for any set $A \subset \overline \U$ and $x\in \overline \U$ let $$p(t,x,A):=\int_{A}p(t,x,y)dy.$$

\noindent
\textbf{Proof of Lemma \ref{convergence1}.} Fix a small positive number $a$ and a large number $T$.
We split $G_n$ in the following way
\begin{eqnarray*}
G_n(x)&=&\frac{2n^2}{|\U_{1,n}|}\int_{0}^{a}\mathbb{P}_x(X(t)\in \U_{1,n})-\mathbb{P}_x(X(t)\in \U_{2,n})\,dt\\
 & + & \frac{2n^2}{|\U_{1,n}|}\int_{a}^{T}\mathbb{P}_x(X(t)\in \U_{1,n})-\mathbb{P}_x(X(t)\in \U_{2,n})\,dt\\
& + & \frac{2n^2}{|\U_{1,n}|}\int_{T}^{\infty}\mathbb{P}_x(X(t)\in \U_{1,n})-\mathbb{P}_x(X(t)\in \U_{2,n})\,dt
\end{eqnarray*}
and similarly, \begin{eqnarray*}
G(x)&=&\frac{2}{\rm{area}(\U_{1})}\int_{0}^{a}p(t,x,\U_{1})-p(t,x,\U_{2})\,dt \\
 & + & \frac{2}{\rm{area}(\U_{1})}\int_{a}^{T}p(t,x,\U_{1})-p(t,x,\U_{2})\,dt\\
& + & \frac{2}{\rm{area}(\U_{1})}\int_{T}^{\infty} p(t,x,\U_{1})-p(t,x,\U_{2})\,dt .
\end{eqnarray*}
\nin
To prove the lemma we will show that the three terms in the RHS for $G_n(x)$ are close to the corresponding terms for $G(x).$
From Lemma \ref{bh} it follows that uniformly over $x\in \overline{\U}$ and $t\ge T$
$$\bigl|p(t,x,\U_{1})-\frac{\rm{ area}(\U_{1})}{\rm{area}(\U)}\bigr|\le e^{-ct}$$ 
and similarly for $\U_{2}.$
Now since by hypothesis $\rm{ area}(\U_{1})=\rm{ area}(\U_{2}),$ 
$$\sup_{x \in \overline \U} \int_{T}^{\infty}|p(t,x,\U_{1})-p(t,x,\U_{2})|\,dt \le 2\int_{T}^{\infty}e^{-ct}dt.$$
Also by  Lemma \ref{welldefined} and our interpolation scheme 
$$
\sup_{x \in \overline \U}\int_{T}^{\infty}|\mathbb{P}_x(X(t)\in \U_{1,n})-\mathbb{P}_x(X(t)\in \U_{2,n})|\,dt \le 2\int_{T}^{\infty}e^{-Ct}dt
$$
for $T$ large enough. Thus in both the expressions the third term can be made arbitrarily small by choosing $T$ large enough.
Now the integral in the first term is at most $a$ and hence can be made small by choosing $a$ small enough. We now show that for any fixed $a$ and $T$ the middle term goes to zero as $n$ goes to infinity. 
First notice that since $|\U_{1,n}|=n^2 \rm{area}(\U_{1})+O(n)$ it suffices to just show, 
$$
\sup_{x\in \overline\U}\left|\int_{a}^{T}\bigl[\mathbb{P}_x(X(t)\in \U_{1,n})-\mathbb{P}_x(X(t)\in \U_{2,n})\bigr]\,dt-\int_{a}^{T}\bigl[p(t,x,\U_{1})-p(t,x,\U_{2})\,dt\bigr]\right|
$$
goes to $0$ as $n=2^m \rightarrow \infty.$
Using  Theorem \ref{clt} we choose  $n$ large enough such that $$\sup_{[a,T]}\sup_{x,y\in \overline{\U}}|p_n(t,x,y)-p(t,x,y)|\le \e$$ for some small number $\e$.
Now the continuity of the heat kernel $p(t,x,y)$ allows us to use Riemann sums to approximate 
$$\int_{a}^{T}[p(t,x,\U_{1})-p(t,x,\U_{2})]\,dt.$$
That is 
\begin{equation}\label{rieman}
\int_{a}^{T}(p(t,x,\U_{1})-p(t,x,\U_{2}))\,dt= \lim_{n\rightarrow \infty}\int_{a}^{T}\bigl[\sum_{y\in \U_{1,n}}\frac{p(t,x,y)}{n^2}-\sum_{y\in \U_{2,n}}\frac{p(t,x,y)}{n^2}\bigr]\,dt.
\end{equation}
More over the above convergence is uniform in $x$ on $\overline \U$. This is because $p(t,x,y)$ is continuous on the compact set $[a,T]\times \overline \U \times \overline \U.$
Now using Theorem \ref{clt} we get that for any fixed $a$ and $T$
\begin{eqnarray*}
\lim_{{n=2^m} \atop {m \rightarrow \infty}}\sup_{x\in \overline\U}\left|\bigl(\int_{a}^{T}[\sum_{\U_{1,n}}\frac{p_n(t,x,y)}{n^2}-\sum_{\U_{2,n}}\frac{p_n(t,x,y)}{n^2}]\,dt\bigr)-\bigl(\int_{a}^{T}[\sum_{\U_{1,n}}\frac{p(t,x,y)}{n^2}-\sum_{\U_{2,n}}\frac{p(t,x,y)}{n^2}]\,dt\bigr)\right|=0.
\end{eqnarray*}
Also by definition for all $x\in \overline \U,$$$\int_{a}^{T}[\sum_{\U_{1,n}}\frac{p_n(t,x,y)}{n^2}-\sum_{\U_{2,n}}\frac{p_n(t,x,y)}{n^2}]\,dt=\int_{a}^{T}[\mathbb{P}_x(X(t)\in \U_{1,n})-\mathbb{P}_x(X(t)\in \U_{2,n})]\,dt.$$
\nin
Thus by \eqref{rieman} 
\begin{equation}\label{limit}
\lim_{{n=2^m} \atop {m \rightarrow \infty}}\sup_{x\in \overline\U}\left|\int_{a}^{T}\bigl[\mathbb{P}_x(X(t)\in \U_{1,n})-\mathbb{P}_x(X(t)\in \U_{2,n})\bigr]\,dt-\int_{a}^{T}\bigl[p(t,x,\U_{1})-p(t,x,\U_{2})\,dt\bigr]\right|=0.
\end{equation}
Hence we have shown that the three terms which we decomposed $G_n(x)$ into at the beginning of the proof can be made arbitrarily close to the corresponding terms for $G(x)$ for large $n$. Thus the proof is complete.
\qed
\\

\noindent
\textbf{Proof of Theorem \ref{convergence}.}
The proof follows immediately from Lemma \ref{convergence1} and Theorem \ref{closedform1}.
\qed

\section{PART II: Hitting measure estimates on $\U_n$} \label{RWE}
\nin
The first result is a Beurling type estimate which says that for any connected subset $A$ of $\U_n$ with large enough diameter which is at a certain distance away from $\U_{1,n}$ the probability that random walk started from a neighboring site of $A$, hits $\U_{1,n}$ before hitting $A$ decays as a power law in $n$.   We first need the following definition.
\begin{definition}\label{conndef1}
Let $\U_n^*$ denote the graph $\U_n$ along with all the diagonals of the squares that are entirely in $\U_n.$ We will call connected subsets of $\U_n^*$ as $*-$ connected subsets of $\U_n.$
\end{definition}
\nin
Also recall the definition of random walk from \eqref{rwhk}. In the sequel for any subset $A\subset \U_n,$ $\tau(A)$ will denote the hitting time for the random walk.
\begin{lem}\label{hitprob2}Fix $c>0$. Consider $A\subset \U_n$ be $*-$connected. Also suppose that $\min( diam(A),d(\U_{1},A))\ge c$ .
Then for large $n$, for all such $A,$ $$\sup_{x\sim A}\mathbb{P}_{x}(\tau(\U_{1,n})\le \tau(A))\le \frac{C}{n^{\beta}}$$ for some positive $\beta,C$ depending only on $c$ and $\U$. Here $x \sim A$ means that $x \notin A$ and there exists $y \in A$ such that $x$ is a neighbor of $y.$
\end{lem}
\nin
The next lemma says that uniformly from any point $z$ at distance $1/2$ (any constant would work) from $x_1$ and all subsets $A\subset \U_n$ of large enough $\pi_{RW}$ measure,  the chance that the random walk does not hit $A$ before reaching a ball of radius $\e$ from $x_1$ goes to $0$ as $\e$ goes to $0$.

\begin{lem}\label{polydecay}Fix a constant $\alpha \in (0,1).$ For all $\U,x_1$ as in Setup \ref{para},  
$$\lim_{\e \to 0}\limsup_{n\rightarrow \infty}\sup_{{A \subset \U_n}\atop{\pi_{RW}(A)\ge \alpha }}\left[\sup_{z\in \U_n \setminus B(x_1,\frac{1}{2})}\mathbb{P}_z\bigl\{\tau(B(x_1,\e))\le \tau({A})\bigr\}\right]=0$$
where $\pi_{RW}(\cdot)$ is the stationary measure of the random walk on $\U_n$ and $B(x_1,\e)$ is the euclidean ball of radius $\e$ around $x_1.$
\end{lem}
The next lemma compares the hitting times of various sets i.e. how do the hitting times of various sets at different distances from the starting point compare.
\begin{lem}\label{gflower}
 Given small enough $\e> 0$,  for all $z \in \U_n \cap \U_{(\alpha-\sqrt{\e})}$,  and $y\in \U_n$ such that $d(y,z) \le \e^2$,
$$\mathbb{P}_{y}(\tau(z)< \tau(\U_n\setminus \{\U_n \cap \U_{(\alpha)}\}))= \Theta(\frac{\log(1/d(y,z))}{\log n}),$$
where the constant in the $\Theta$ notation depend on $\e, \alpha, \U.$
\end{lem}

\begin{remark}
One can prove the above lemma from conductance estimates on the graph $\U_n$. However for the sake of unification, all the proofs in this article will use heat kernel estimates for the random walk on $\U_n$ which we state soon.
\end{remark}
As mentioned before the above estimates are standard for the 
random walk on the entire lattice $\mathbb{Z}^2$. We make necessary adaptations to obtain the results for the random walk on the bounded geometry $\U_n.$
We start by stating a basic mixing time result for the random walk on $\U_n.$
\begin{lem}\label{lmt}Given a smooth domain $\U$ as in Setup \ref{para} and $\e>0$ there exists a constant $C=C(\e,\U)$ such that for large enough $n,$ $$t_{\infty}(\e)\le C$$ where $t_{\infty}(\cdot)$ is the $\mathbb{L}_{\infty}$ mixing time for the random walk on $\U_{n}$. 
\end{lem}
\begin{proof} 
The fact is standard and the proof  follows from \cite[Theorem 13]{morp} and the isoperimetric inequality proved in \cite[Theorem 5.5]{wtf}.
\end{proof}
\noindent
We now state a standard property about the boundary of $\U$ as in Setup \ref{para}. Since the boundary $\partial \U$ is analytic there exists a $C>0$ and an $\e_0$ such that for all $x\in \partial \U$ there exists an orthogonal system of coordinates centered at $x=(x_1,x_2)$ such that for all $\e\le \e_0$ 
\begin{equation}\label{localhalf}
B(x,\e)\cap \U=\left\{ (x_1',x_2') \in B(x,\e): x_1'\in (x_1-\e,x_1+\e),x_2'\ge f(x_1') \right\}
\end{equation}  
 and $$|f(x_1')-x_2|\le C|x_1'-x_1|^2.$$ The above  is a simple consequence of Taylor expansion up to second order of the curve locally near $x$.  See Fig \ref{fig.local}.

\begin{figure}
\centering
\includegraphics[scale=.6]{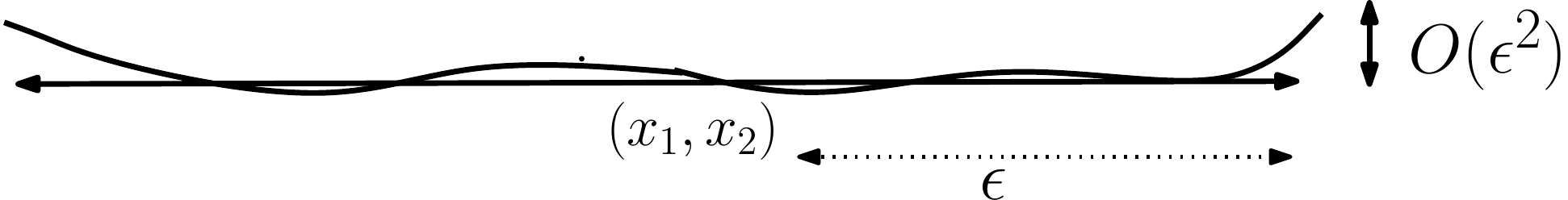}
\caption{Locally the region near the boundary looks like a half plane}
\label{fig.local}
\end{figure}

\noindent
As a simple corollary of the above fact we see that $\U$ satisfies the following property which shows that the $y_i$'s  in Setup \ref{para} can indeed be chosen. 
\begin{cor}\label{ias1}
Let $\U$ be as in Setup \ref{para}. Then there exists $\dd_0=\dd_0(\U)$ such that for all $x\in \overline{\U}$ and $\dd<\dd_0$ there exists $y\in \U$ such that 
\begin{eqnarray}\label{ia1}
d(y,x)&\le & \dd\\
\label{ia2}
B(y,\dd/2)&\subset & \U.
\end{eqnarray}
\end{cor}
\noindent
Recall that $d(y,x)$ is the euclidean distance between $x$ and $y$. $B(y,\dd)$ denotes the euclidean ball of radius $\dd$ with center at $y.$\\

\noindent
\begin{proof}Choose  $\dd_0\le \e_0/4$ such that $$C\dd_0^{2}\le \frac{\dd_0}{100}$$ where $\e_0$ and $C$ appear in \eqref{localhalf}. For any $\dd< \dd_0$ the lemma is immediate if $d(x,\partial \U)>\dd/2$. since then we can choose $y=x.$ \\
\noindent
Otherwise let $z=(z_1,z_2)\in \partial U$ be the closest point on the boundary to $x$. 
Now in the local coordinate system centered at $z$ as in \eqref{localhalf} choose $y=(z_1,z_2+\dd/2)$.
Then $$d(x,y)\le d(x,z)+d(z,y)\le \dd.$$ Also clearly $B(y,\dd/2)\subset \U$  and hence we are done. See Fig \ref{fig.iap}.
\end{proof}

\noindent
For the remaining part we need  the  following gaussian upper and lower bounds on the heat kernel of random walk on $\U_n$. Recall the definition of $p_n(t,x,y)$ from \eqref{normkernel}.
\begin{thm}\label{wtfest}Given $\U$ as in Setup \ref{para},
\begin{itemize}
\item [i.]
\label{lclt1}\cite[Theorem 2.9]{wtf}
for any $T$ there exists  $C_1$ and $C_2$ such that for all $t\in[\frac{1}{n},T],$
\begin{equation}\label{lclt2}
p_n(t,x,y)\le \frac{C_1}{(t^{1/2}\vee\frac{1}{n})^2} \exp(-C_2\frac{d(x,y)^2}{t}), 
\end{equation}\\
\item [ii.]
\cite[Cor 2.2.5]{thes}\label{exittime} for any $T>0$ there exists $C=C$ and $N$ such that 
\begin{equation}\label{et}
\mathbb{P}_{x}[\sup_{s\le t}d(X_s,x)\ge \eta]\le C\exp\left(t-\frac{\eta}{4(\frac{1}{n}\lor t^{1/2})}\right)
\end{equation}
for all $n> N$, $t\le T,\,\,x\in \U_n,\,\, \eta>0$,\\
\item [iii.] \cite[Theorem 2.10]{wtf} \label{glb1}there exists $C_1,C_2$ dependent only on $T$ such that 
\begin{equation}\label{glb}
p_{n}(t,x,y)\ge \frac{C_1}{(\frac{1}{n}\vee t^{1/2})^2}\exp\left(-C_2\frac{d(x,y)^2}{t}\right)
\end{equation}
for all $t\le T$ and $x,y \in \U_n.$
\end{itemize}
Note that all the constants above implicitly depend on $\U.$
\end{thm}
\begin{figure}
\centering
\includegraphics[scale=.5]{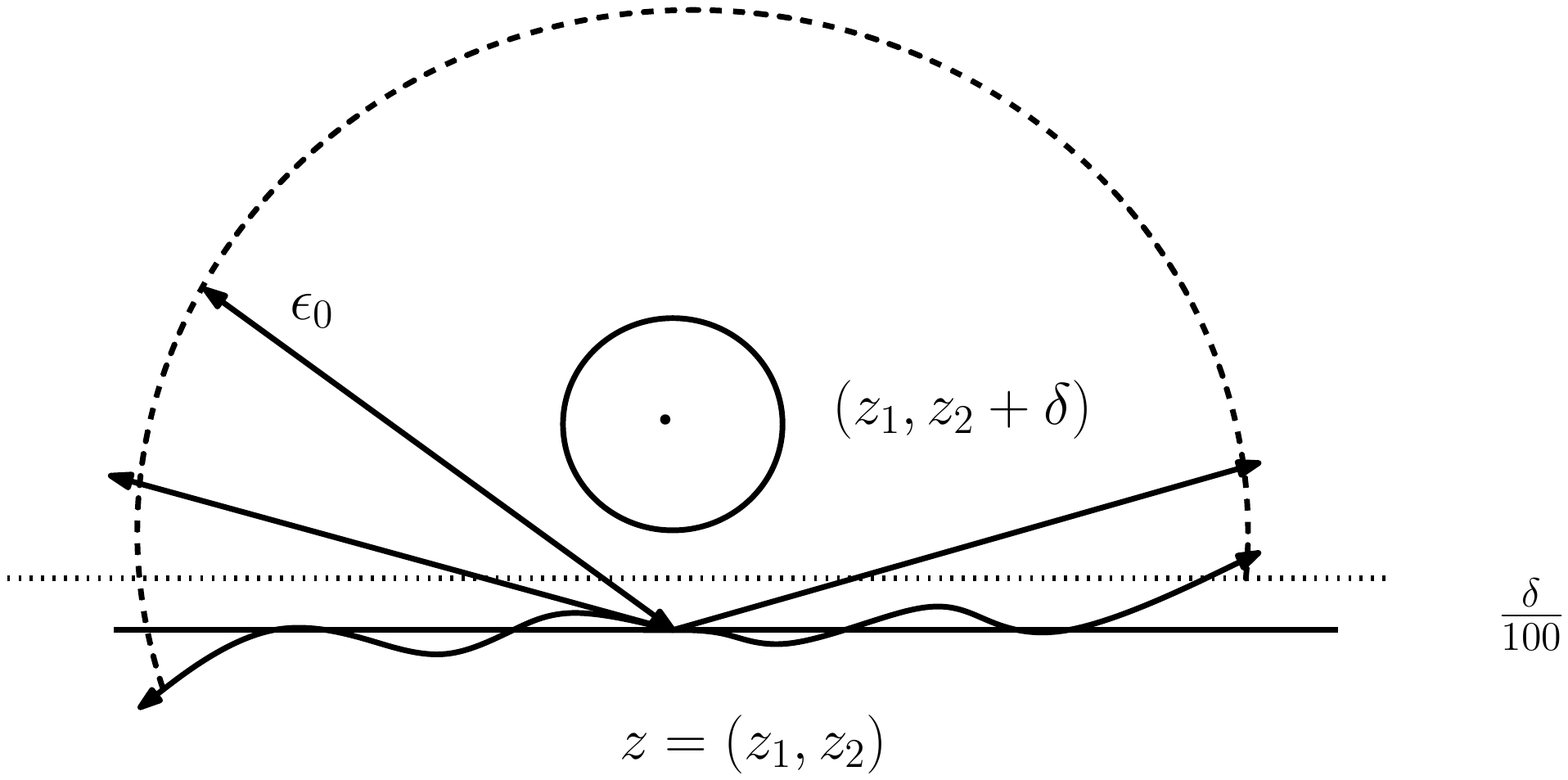}
\caption{Illustrating the proof of Corollary \ref{ias1}.}
\label{fig.iap}
\end{figure}
\nin
We now proceed towards proving Lemma \ref{hitprob2}. To this end we need a preliminary result.
We start with a definition.
Fix any $x\in \U_n$. Let us consider concentric discs $D_{x,j}$ around $x$ of radius $\frac{2^j}{n}$ i.e.  
\begin{equation}\label{shell1}
D_{x,j}:=B(x,\frac{2^j}{n})\cap\U_n.
\end{equation}
Recall $\partial_{out}$ from \eqref{bdry1}.
Let  
\begin{equation}\label{shellbdr}
C_{x,j}:=\partial_{out}D_{x,j}.
\end{equation}
\nin
We now show that starting from any point in $\U_n$ which is $\frac{1}{n}$ distance away from $C_{x,j-1}$ there is a constant chance of the random walk moving away from the boundary of $\U$, before hitting $C_{x,j}$. That is, it is unlikely that the random walk path on $\U_n$ hitting $C_{x,j}$ from $C_{x,j-1}$  stays uniformly close to the boundary of $\U$.
The next result makes the above statement precise. Let $\e< \dd_0$ where $\dd_0$ appears in the statement of Corollary \ref{ias1}.
Now by Corollary \ref{ias1} for every $j\le \log(\dd_0 n)$ and all $z\in C_{x,j-1}$, there exists a $y$ such that $$d(z,y)\le\frac{\e2^{j-1}}{n},\mbox{ and }   B\bigl(y,\frac{\e2^{j-1}}{2n}\bigr) \subset \U.$$
Define,
\begin{equation}\label{shortnote1}
B_{z,j-1}:=B\bigl(y,\frac{\e2^{j-1}}{4n}\bigr).
\end{equation} 
For notational brevity we choose to suppress the $\e$ dependence in  the notation above. 
\begin{lem}\label{chanceinterior}There exists constants $\e,c$ such that for all $x\in \U_n,$  for any $z\in C_{x,j-1}$ with $j=1, \ldots, \log (\dd_0n)$,
$$\mathbb{P}_{z}(\tau (B_{z,j-1})<\tau(C_{x,j}))\ge c,$$
 where $\dd_0$ appears in the statement of Corollary \ref{ias1}.
\end{lem}
See Fig. \ref{f.pathlike}.
\begin{figure}
\center
\includegraphics[scale=.8]{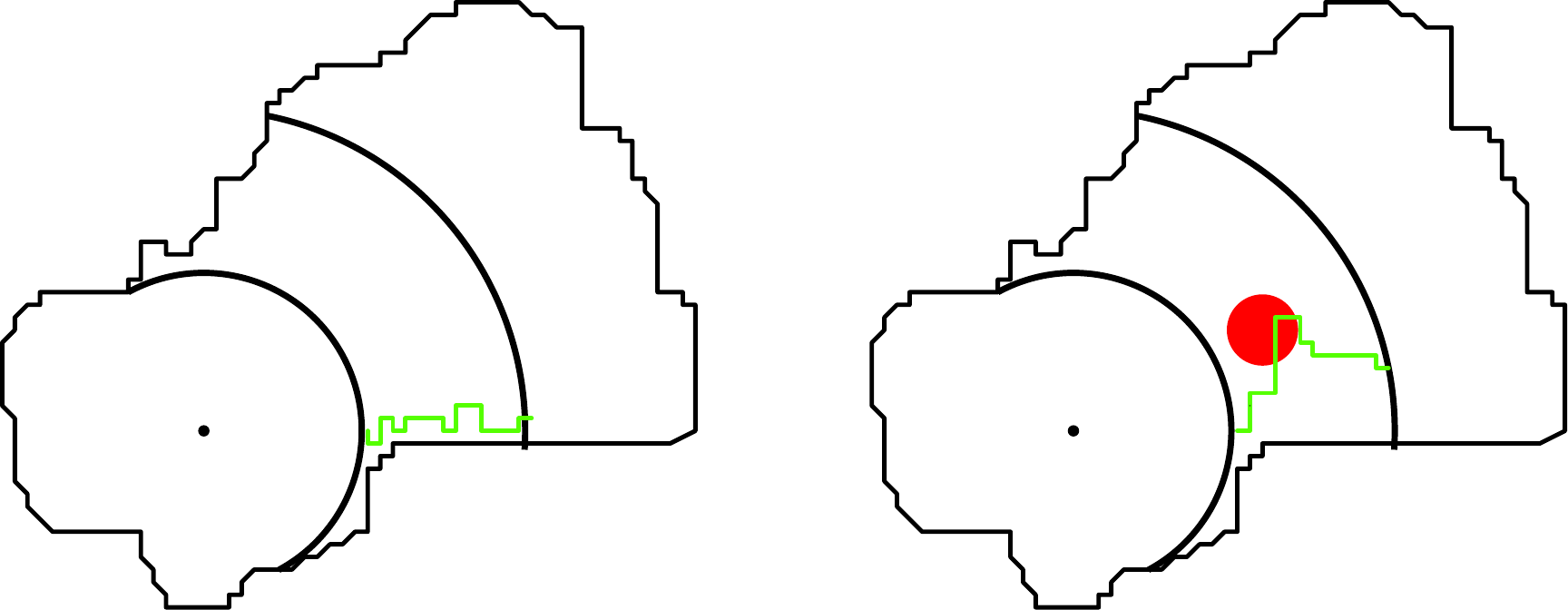}
\caption{The left hand side shows the unlikely event that a path from $C_{x,j-1}$ hits $C_{x,j}$  and the path stays close to the boundary of $\U$. The right hand side shows a typical path. }
\label{f.pathlike}
\end{figure}
\\\\
\begin{proof}To prove the lemma we use Theorem \ref{wtfest} $ii.$ and $iii$.
Since for all $z\in C_{x,j-1},$ $d(z,C_{x,j})\ge \frac{2^{j-1}}{n}$ taking $T=1$, $t=\e^2\frac{2^{2j}}{n^2}$ and $\eta=\frac{2^{j-1}}{n}$, by \eqref{et} we have 
$$\mathbb{P}_{z}[\tau({C_{x,j}})\le t ]\le Ce\exp\left(-\frac{1}{4\epsilon}\right).
$$
Now we look at the chance that $\tau(B_{z,j-1})\le t.$ Clearly 
$$\mathbb{P}_{z}(\tau({B_{z,j-1}})\le t)\ge \mathbb{P}_{z}(X(t)\in B_{z,j-1}).$$
Let us recall from \eqref{normkernel} that by definition $$\mathbb{P}_z(X(t)\in B_{z,j-1})=\sum_{y \in B_{z,j-1}}p_n(t,z,y)m_n(y).$$ 
Now since all the points in $B_{z,j-1}\cap \U_n$ have $4$ neighbors $m_n(y)=\frac{1}{n^2}$ for all $y \in B_{z,j-1}.$ Let us for the moment denote it by $m.$
By \eqref{glb} we get 
\begin{eqnarray*}
\mathbb{P}_{z}(X(t)\in B_{z,j-1})\ge C_1m\frac{|B_{z,j-1}\cap \U_n |}{ t}\exp\left( -C_2  \right)=\frac{{\epsilon}^2 2^{2(j-1)}}{\pi n^2\epsilon^2 2^{2j}/n^2}= \Theta(1).
\end{eqnarray*}
\nin
Thus 
\begin{eqnarray*}
\mathbb{P}_{z}(\tau({B_{z,j-1}})<\tau(C_{x,j}))&\ge & \mathbb{P}_{z}(\tau(B_{z,j})<t)-\mathbb{P}_{z}(\tau(C_{x,j})<t)\\
&\ge & \Theta(1)-O(\exp(-\frac{1}{4\e})).
\end{eqnarray*} 
Thus we are done by choosing $\e $ small enough.
\end{proof}

\begin{figure}[hbt]
\centering
\includegraphics[scale=.8]{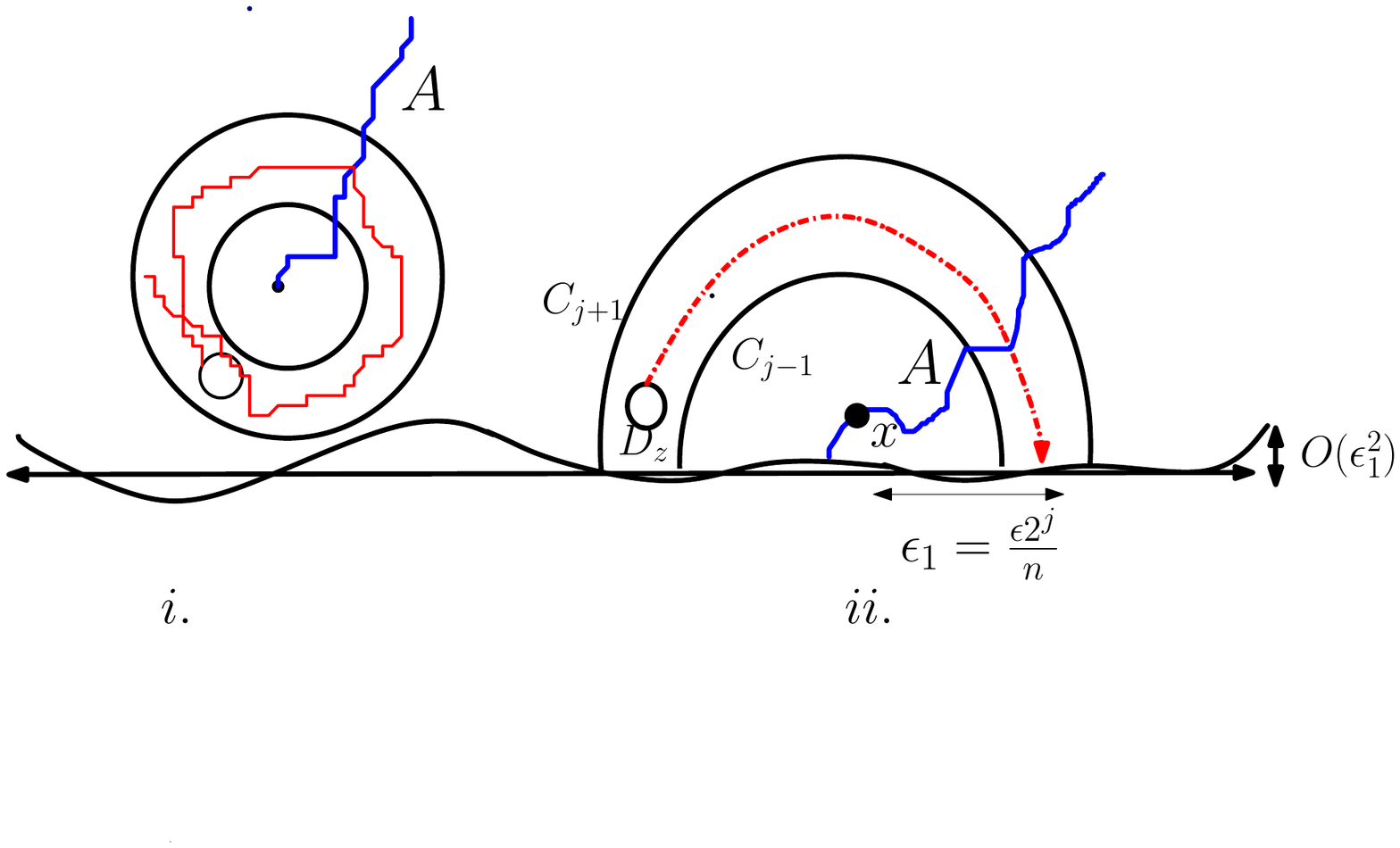}
\caption{Random walk completing a full circle oriented in one direction will hit the set $A.$}
\label{fig.ill1}
\end{figure}
\nin
We are now ready to prove Lemma \ref{hitprob2}. As remarked earlier the basic structure of the arguments are standard and are used to prove similar statements on the whole lattice. We make the necessary additional arguments to prove the statement in the bounded geometry of $\U_n$. \\ 

\noindent
\textbf{Proof of Lemma }\ref{hitprob2}. We will fix $\dd_1 <\dd_{0}/4\wedge  c$ where $\dd_0$ appears in Corollary \ref{ias1} and $c$ appears in the statement of the lemma. Let $x \sim A$.
Recall the definitions of $D_{x,j}$ and $C_{x,j}$ from \eqref{shell1} and \eqref{shellbdr} with center $x$. 
Let us denote by $$I_j:=D_{x,j}\setminus D_{x,j-2}.$$ 
for $j=1 \ldots \log n\dd_1 $.
Now given any $x$ if $d(x,\partial \U)\ge \dd_1$ then for all $j,$ $I_j$ is an annulus. 
Otherwise some of the $I_{j}$'s are a topological quadrilateral and two of the four sides are a part of $\partial \U$. See Fig \ref{fig.ill1} ii. If $\dd_1$ is chosen to be small enough clearly these are the only two possibilities. For $z \in C_{x,j}$ recall $B_{z,j}$ from \eqref{shortnote1}. 
By Lemma \ref{chanceinterior} there is a positive $c_1$ such that for all $j=1,\ldots, \log(\dd_1n), $
\begin{equation}\label{lower2}
\inf_{z\in C_{x,j}}\mathbb{P}_z(\tau(B_{z,j})\le \tau(C_{x,j+1}))\ge c_1.
\end{equation}
\noindent
Now any point in $B_{z,j}$ is in the interior of $I_j$ and at distance at least $\frac{\e 2^{j-1}}{4n}$ from $C_{x,j+1}\cup C_{x,j-1}\cup \partial \U.$ Let $z'$ be any such point.
Since by hypothesis $A$ is a $*-$connected set and $diam(A)>c$ 
the following observations are straightforward corollaries of the Jordan curve theorem:\\
for $j=1 \ldots \log(\dd_1n),$
\begin{itemize}
\item   
If $I_j$ is an annulus:  any curve starting from $z'$ which stays $\frac{\e}{100}\frac{2^{j-1}}{n}$ away from the boundary, makes a full circle and completes a closed loop hits $A$. Fig \ref{fig.ill1} $i.$\\
\item 
If $I_j$ is a topological quadrilateral: Out of the four sides, two sides of $I_j$ are a part of $\partial\U.$  There exists one of these two sides  such that any curve $\gamma$ which starts from $z'$ and stays $\frac{\e}{100}\frac{2^{j-1}}{n}$ away from the three sides before hitting that side hits $A.$ Fig \ref{fig.ill1} $ii.$  \\
\end{itemize}
\noindent
By the Donsker invariance principle, simple random walk does both the above things with constant probability only dependent on $\e $ and independent of $j$. 
Thus the chance that random walk started from $x$ hits $A$ between $\tau(C_{x,j})$ and $\tau({C_{x,j+1}})$ is at least $d$ for some constant $d=d(\e)$. This follows since from the location $z$ at time $\tau(C_{x,j})$ the random walk with constant chance hits $B_{z,j}$ and from there hits $A$ with constant probability before $\tau(C_{x,j+1})$.

\noindent
We consider the time interval $[\tau(C_{x,j}),\tau(C_{x,j+1}))$ as the $j^{th}$ round. By the previous discussion in each round the chance to hit $A$ is at least $d$ for $j=1 \ldots \log (n\dd_1)$.
Also since by hypothesis $d(A,\U_1)\ge c$ and $x\sim A,$  $$\tau(C_{x,j})<\tau(\U_1)$$ for all such $j.$
Hence $$\sup_{x\sim A}\mathbb{P}_{x}\left(\tau(\U_1)\le \tau(A)\right)\le (1-d)^{\log(\dd_1 n)}.$$
\qed
\\\\
\noindent
We now state and prove another similar lemma. Recall the definitions of $y_1,\U_{1,n}$ from Setup \ref{para}.
\begin{lem}\label{capa}Let $0<\e_1< \e_2.$ Assume $A\subset \U_n$ is a connected set such that  
 $$d(\U_{1,n},A)\le \e_1.$$ Also assume  $A \cap \bigl(\U_n\setminus B(y_1,\e_2)\bigr)\neq \emptyset.$ 
Then $$\sup_{x\in \U_{1,n}}\mathbb{P}_x(\tau(\U_n\setminus B(y_1,\e_2))\le \tau(A))\le {C^{\log{\left(\frac{\e_2}{\e_1}\right)}}}, $$
for some $C=C(\U)<1$ independent of $n$.
\begin{figure}[hbt]
\centering
\includegraphics[scale=.5]{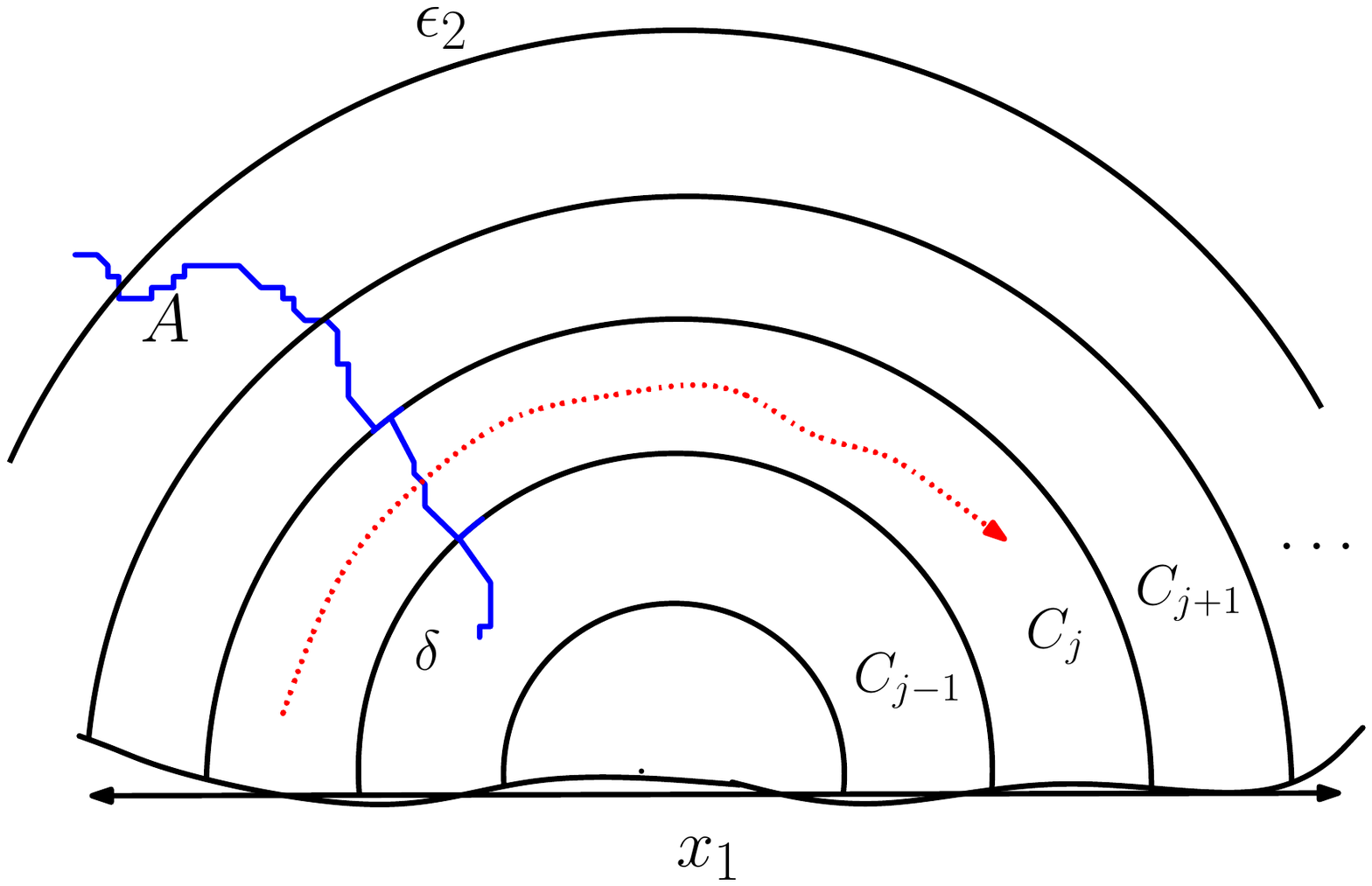}
\caption{Figure illustrating the proof of Lemma \ref{capa}. The red random walk path making a loop hits $A$.}
\label{f.ill12}
\end{figure}
\end{lem}
\nin
The statement of the lemma roughly says if a connected set $A$ of large enough diameter is close enough to $\U_{1,n}$ then random walk starting from $\U_{1,n}$ is more likely to hit the set $A$ before exiting a large enough ball.\\
\noindent
\begin{proof}The proof of this lemma is similar to the proof of Lemma \ref{hitprob2}. We look at shells of exponentially growing radii centered at $y_1$  i.e. $B(y_1,2^j\e_1)\bigcap \U$ for $j=1\ldots \log(\frac{\e_2}{\e_1}).$ Recall $\partial_{out}$ from \eqref{bdry1}. Let $$C_j=\partial^{out} \bigl(B(y_1,2^j\e_1)\bigcap \U_n\bigr).$$ 
Let $\tau(C_j)$ be the first time that the random walk hits $C_j.$ We first claim that there is a constant $c$ such that for all $j$,
$$\inf_{z\in C_j}\mathbb{P}_{z}(\tau(A)\le \tau(C_{j+1}))\ge c.$$
The proof of the above claim is the same as the proof of Lemma \ref{hitprob2}. We omit the arguments to avoid repetition.\\

\noindent
Now let the random walk start anywhere from $\U_{1,n}$ and let $z\in C_j$ be the point it hits at $\tau(C_j).$
Thus again using the round argument as in the proof of the previous lemma 
 $$\sup_{x\in \U_{1,n}}\mathbb{P}_x\bigl(\tau (\U\setminus B(x_1,\e_2))\le \tau_{A}\bigr)\le (1-c)^{\log{\frac{\e_2}{\e_1}}}.$$
 Note that the number of rounds here is $\log{\frac{\e_2}{\e_1}}$.
\end{proof}
\nin
\textbf{Proof of Lemma \ref{polydecay}}.
We first prove the following lemma:\\
\begin{lem}\label{hitprob1} Fix $\alpha \in (0,1)$. Then there exists  constants $C,D$ such that for  all large enough $t$,
$$\sup_{{A \subset \U_n}\atop{\pi_{RW}(A)\ge \alpha }}\sup_{x\in \U_{n}} \mathbb{P}_x  (\tau(A) \ge t) \le Ce^{-Dt}$$
where $\pi_{RW}(\cdot)$ is the stationary measure of the random walk on $\U_n.$
\end{lem}
\nin
This is a standard mixing result which says that the hitting time of any set $A$  of large enough measure, for the random walk has exponential tail .\\
\noindent
\begin{proof}Fix any $A\subset \U_n$. 
As stated in Lemma \ref{lmt}  $\mathbf{t}:=t_{\infty}(\frac{1}{4})=O(1).$   
Now since $\pi_{RW}(A)\ge \alpha $, $\displaystyle{\inf_{x\in \U_n}\mathbb{P}_{x}(X_{\mathbf{t}}\in A)\ge \alpha/4},$
and hence 
 $$\inf_{x\in \U_n}\mathbb{P}_{x}(\tau(A)\le {\mathbf{t}})\ge \alpha/4.$$
Therefore for any $t>0$ $$\mathbb{P}_{x}(\tau(A)\ge t)\le (1-\alpha/4)^{\lfloor{T/{\mathbf{t}}}\rfloor}.$$
This is because from any $y\in \U_n$ there is a chance of at least $\alpha/4$ to hit $B_1$ in the next time interval of length ${\mathbf{t}}$. Hence the lemma is proved.
\end{proof}
\nin
We resume the proof of Lemma \ref{polydecay}.
To prove this we compute the time spent in $B(x_1,\e)$ starting from a point $z_1$ and $z_2$ such that $d(x,z_1)\le 2\e$ and $d(x,z_2)>\frac{1}{2}.$ This is helpful because of the following formula:
\begin{equation}\label{markov34}
\mathbb{P}_{z_2}(\tau(B(x_1,\e)<T))\le \frac{\int_{0}^{2T}{\mathbb{P}_{z_2}(t,B(x_1,\e))}}{\inf_{z_1}\int_{0}^{T}{\mathbb{P}_{z_1}(t,B(x_1,\e))}}. 
\end{equation}
The above follows by markov property.
The goal now is to prove upper and lower bounds on the numerator and the denominator respectively. 
We use the upper and lower bound on the gaussian heat kernel stated in Theorem \ref{wtfest}. Assume $T\ge 1.$
 We see that 
\begin{align*}
\int_{0}^{T}\mathbb{P}_{z_1}(t,B(x_1,\e))& \ge \int_{\e}^{T}[\sum_{y\in B(x_1,\e)}\frac{1}{n^2}p_{n}(t,z_1,y)]dt\\
&\ge \int_{\e}^{1}[\sum_{y\in B(x_1,\e)}\frac{1}{n^2}p_{n}(t,z_1,y)]dt\\
&\ge C\e^2 \int_{\e}^{1}\frac{1}{t}dt\\
&\ge C \e^2\log(\frac{1}{\e}).
\end{align*}
The second last inequality follows from \eqref{glb} ( the constant $C$ is  changing from line to line). 
We now look at the numerator.  Recall $\mathbf{t}$ from the proof of the last lemma.
\begin{align*}
\int_{0}^{2T}\mathbb{P}_{z_2}(t,B(x_1,\e))& \le \e P_{z_2}(\tau(B(x_1,\e)\le \e)+\int_{\e}^{2T}[\sum_{y\in B(x_1,\e)}\frac{1}{n^2}p_{n}(t,z_2,y)]dt\\
&= \e P_{z_2}(\tau(B(x_1,\e)\le \e)+\int_{\e}^{\mathbf{t}}[\sum_{y\in B(x_1,\e)}\frac{1}{n^2}p_{n}(t,z_2,y)]dt+\int_{\mathbf{t}}^{2T}[\sum_{y\in B(x_1,\e)}\frac{1}{n^2}p_{n}(t,z_2,y)]dt\\
&\le \e O(e^{\frac{-C}{\sqrt{\e}}})+ C\e^2 \int_{\e}^{\mathbf{t}}\frac{1}{t}e^{-\frac{C}{t}}dt+ C\e^2 T\\
&\le C T \e^2.
\end{align*}
The bounds on the first and second terms follow from  \eqref{et} and \eqref{lclt2} respectively. 
Note for the last integral we use the fact that for any $t\ge \mathbf{t}$ ,$p_{n}(t,z_2,y)\le \frac{C}{n^2}.$
Thus using \eqref{markov34}
\begin{align*}
\mathbb{P}_{z_2}(\tau(B(x_1,\e)<T))\le O(\frac{T}{\log(\frac{1}{\e})}).
\end{align*}
Now by Lemma \ref{hitprob1}  
\begin{align*}
\mathbb{P}_{z_1}\{\tau(A)<T)\}\ge 1- O(e^{-DT}).
\end{align*}
Using  the above bounds, taking $z=z_2$ we get that  for any $T \ge 1 $ 
\begin{align*}
\mathbb{P}_z\bigl\{\tau(B(x_1,\e))\le \tau({A})\bigr\} & \le \mathbb{P}_z\bigl\{\tau(B(x_1,\e))<T\bigr\}+\mathbb{P}_z\bigl\{\tau(A)\ge T\bigr\}\\
&\le O\left(\frac{T}{\log(\frac{1}{\e})}\right)+O(e^{-DT}).
\end{align*}
Thus we are done by choosing $T=\sqrt{\log(\frac{1}{\e})}.$
\qed \\

\nin
\textbf{Proof of Lemma \ref{gflower}.}
The proof uses similar arguments as above. We begin by providing sharp upper and lower bounds for the two following quantities: 
\begin{align*}
\E_y[\int_{0}^{\tau}{\bf}{1}(X(t)=z)]\\
\E_z[\int_{0}^{\tau}{\bf}{1}(X(t)=z)] 
\end{align*}
where $\tau=\tau(\U_n \cap[ \U\setminus \U_{(\alpha)}]).$
To show this we first notice that by either Lemma \ref{lmt} or Theorem \ref{wtfest} iii.,  there exists a constant $c=c(\U,\alpha)$ for any $w\in \U_n,$
$$\P_w(\tau \le 1)\ge c.$$
Using Theorem \ref{wtfest} we get for any $w\in \U_n$,
\begin{align}\label{kybnd45}
\int_{0}^{2}\P_w(X(t)=z)  \le \frac{1}{n^2}\int_{0}^{2}\frac{e^{-\frac{d(w,z)^2}{t}}}{\frac{1}{n^2} \vee t} dt = O(\frac{\log(1/d(w,z))}{n^2}).
\end{align}
Putting the above together, 
\begin{align*}
\E_w[\int_{0}^{\tau}{\bf}{1}(X(t)=z)] &\le \int_{0}^{2}\P_w(X(t)=z) + \sum_{\ell=0}^{\infty}\P_w(\tau \ge \ell) \sup_{w \in \U_n}\int_{\ell+1}^{\ell+2}\P_w(X(t)=z) \\
& \le O(\frac{\log(1/d(w,z))}{n^2}) +\sum_{\ell=0}^{\infty}  (1-c)^{\ell} \sup_{w \in \U_n}\int_{1}^{2}\P_w(X(t)=z)\\
&= O(\frac{\log(1/d(w,z))}{n^2}) +O(\frac{1}{n^2}).
\end{align*}
The second term follows by  the fact that starting from any $w$ in time $1$ the random walk is reasonably mixed (by Theorem \ref{wtfest}) and then from the location of the random walk at time $1$ we use \eqref{kybnd45}.
For the lower bound we use strong Markov Property, 
\begin{align*}
\E_y[\int_{0}^{\tau}{\bf}{1}(X(t)=z)] &  \ge \int_{0}^{\e^2}\P_y(X(t)=z)- \P_y(\tau\le  \e^2) \left[\sup_{w: d(w,z)>\sqrt \e}\int_{0}^{\e^2}\P_w(X(t)=z)\right] 
\end{align*}

Now for any $y \in \U_n$ such that $d(y,z) \le \e^2,$
\begin{align*}
\int_{0}^{ \e^2}\P_y(X(t)=z) &\ge \frac{1}{n^2}\int_{0}^{ \e^2 }\frac{1}{\frac{1}{n^2} \vee t}e^{-\frac{d(y,z)^2}{t}} dt\\
& \ge \frac{1}{n^2}\int_{d(y,z)^2}^{\e^2}\frac{1}{t}e^{-1} dt\\
& \ge \Omega(\frac{\log(1/(d(y,z)))}{n^2})
\end{align*}
where the last inequality follows since $d(y,z) \le \e^2$ by hypothesis.

By  Theorem \ref{wtfest},
 $\P_y(\tau\le \e^{2})\le O(e^{- \frac{1}{\sqrt \e}}).$ Also for any $w$ such that $d(w,z) \ge \sqrt \e$  we have 
\begin{align*}
\int_{0}^{\e^2}\P_w(X(t)=z) \le \frac{1}{n^2}\int_{0}^{1/n^2}\frac{1}{\frac{1}{n^2}} + \frac{1}{n^2}\int_{0}^{\e^{2}} \frac{e^{-\frac{\e}{t}}}{t}dt  
\end{align*}

Thus from the above we see that for all $y$ such that $d(y,z) \le \e^2$ we have,
$$
P_y(\tau(z) < \tau(\U\setminus \U_{\alpha})) = \frac{\int_{0}^{\tau}P_y(X(t)=z)}{\int_{0}^{\tau}P_z(X(t)=z)}  = \frac{1}{\log(n)}\Theta(\log(1/d(x,z))),$$ and we are done.
\qed

\bibliographystyle{plain}
\bibliography{GFF}

\end{document}